\documentclass[12pt,a4paper]{amsart}

\usepackage{euscript}
\usepackage{eufrak}
\usepackage{amsmath}
\usepackage{amsthm}
\usepackage{amssymb}
\usepackage{enumerate}
\usepackage{amscd}
\usepackage{epic}
\usepackage{eufrak}
\usepackage[bookmarksnumbered,bookmarksopen,pdfusetitle,unicode]{hyperref}

\numberwithin{equation}{section}
\numberwithin{equation}{subsection}

\theoremstyle{plain}
\newtheorem{theorem}[equation]{Theorem}
\newtheorem{lemma}[equation]{Lemma}
\newtheorem{proposition}[equation]{Proposition}
\newtheorem{corollary}[equation]{Corollary}
\newtheorem{definition}[equation]{Definition}

\theoremstyle{definition}

\newtheorem{remark}[equation]{Remark}

\newcommand{\mo}{\mathfrak{m}}

\def\Q{{\mathbb Q}}
\def\Z{{\mathbb Z}}

\newcommand{\cal}{\mathcal}

\begin{document}

\title[Invariants of open books]{Invariants of open books of links
of surface singularities}

\author[A. N\'emethi]{Andr\'as N\'emethi}
\thanks{The first author is partially supported by OTKA grants,
the second author by Galatasaray University Research fund, and both authors by  the BudAlgGeo
project, in the framework of the European Community's `Structuring the
European Research Area' programme. }
\address{R\'enyi Institute of Mathematics,  1053 Budapest,
Re\'altanoda u. 13--15,  Hungary} \email{nemethi@renyi.hu}
\urladdr{http://www.renyi.hu/\textasciitilde nemethi}
\author[M. Tosun]{Meral Tosun}
 \address{Galatasaray University, Departement of Mathematics, 34257 Ortakoy-Istanbul, Turkiye}
\email{mtosun@gsu.edu.tr}
\urladdr{http://math.gsu.edu.tr/tosun/index.html}



\date{}
\maketitle

\begin{abstract}
If $M$ is the link of a complex normal surface singularity, then it carries a canonical contact structure
$\xi_{can}$, which can be identified from the topology of the 3--manifold $M$. We assume that $M$ is a
 rational homology sphere. We compute the {\em support
genus}, the {\em binding number} and the {\em norm} associated with the open books which support  $\xi_{can}$, provided that we restrict ourself to the case of (analytic) Milnor open books. In order to do this, we determine monotoneity properties of the genus and the Milnor number of all Milnor fibrations in terms of the Lipman cone.

We generalize  results of \cite{Mohan} valid for links of rational
surface singularities, and we answer some  questions of  Etnyre and Ozbagci \cite[section 8]{burak-etnyre}
 regarding the above invariants.
\end{abstract}

\section{Introduction}

Let $M$ be an oriented 3-dimensional manifold. By a result of Giroux
\cite{Gi}  there is a one-to-one correspondence between
open book decompositions  of $M$ (up to stabilization) and contact structures
on $M$ (up to isotopy). In \cite{burak-etnyre} Etnyre and Ozbagci
consider three invariants associated with a fixed contact structure $\xi$ defined in terms of
all open book decompositions supporting it:

\vspace{1mm}

$\bullet$ the {\it support genus} sg$(\xi)$ is the minimal possible genus for
a page of an open book that supports $\xi$;

\vspace{1mm}

$\bullet$ the {\it binding number} bn$(\xi)$ is the minimal number of of
binding components for an open book supporting $\xi$ and that has pages of
genus sg$(\xi)$;

\vspace{1mm}

$\bullet$ the {\it norm} $\mathfrak{n}(\xi)$ of $\xi$
is the negative of the maximal (topological)
Euler characteristic of a  page of an open book that supports $\xi$.

\vspace{1mm}

In the present article we  determine and characterize completely
the above invariants under the following restrictions:
$M$ will be a \textit{rational homology sphere
which can be realized as the link of a complex surface singularity} $(S,0)$. Moreover, we will
restrict ourselves to the collection of those open book decompositions which
can be realized as Milnor fibrations determined by some analytic germ
(the so-called Milnor open books).
Notice that by \cite{CP}, all the Milnor open book  decompositions define the
same contact structure on $M$, the {\it canonical contact structure} $\xi_{can}$. This structure is
also induced by {\em any} complex structure $(S,0)$ realized on the topological type, and it
can be characterized completely from the topology of $M$.

Hence our results will be applied
exactly for the canonical  contact structure $\xi_{can}$, and for (analytic) Milnor open books,
cf. section \ref{s5}. The corresponding invariants are denoted by
sg$_{an}(\xi_{can})$, bn$_{an}(\xi_{can})$ and $\mathfrak{n}_{an}(\xi_{can})$.

The present article  generalize  results of \cite{Mohan} valid for links of rational
surface singularities, and we answer some  questions of  \cite[section 8]{burak-etnyre}
 regarding the above invariants.

\section{Preliminaries}

\subsection{Invariants associated with a resolution.}
In what follows we assume that $(S,0)$ is a complex normal surface singularity
whose link is a rational homology sphere. Let $\pi :X\longrightarrow S$ be a
good resolution. We will denote by $E_1,\ldots ,E_n$ the smooth irreducible
components of the exceptional curve $E:=\pi^{-1}(0)$ and by $\Gamma $ its dual
graph. By our assumption,  each $E_i$ has genus 0 and $\Gamma $ is a tree.

Consider the free group ${\mathcal G}:=H_2(X,\Z)$ generated by the irreducible
components of $E$, i.e. ${\mathcal G}=\lbrace D=\sum_{i=1}^{n}
m_iE_i \mid m_i\in \Z \rbrace $. On ${\mathcal G}$ there is a natural
intersection pairing $(\cdot,\cdot)$ and a natural partial ordering:
$\sum_i m'_iE_i\leq \sum_i m''_iE_i$ if and only if $m'_i\leq m''_i$ for all
$i$.

We denote the Lipman cone (semi-group) by
$${\mathcal E}^+=\lbrace D\in{\mathcal G}
\mid (D, E_i)\leq 0 \ \ \mbox{for any $i$} \rbrace. $$
It is known (see e.g. \cite{A,Lip})
that if $D=\sum m_iE_i\in{\mathcal E}^+$ then
$m_i\geq 0$ for all $i$, and $m_i>0$ for all $i$ whenever
$D\in{\mathcal E}^+ \setminus \{0\}$.
Moreover, ${\mathcal E}^+ \setminus \{0\}$
admits a unique minimal element (the so-called Artin, or fundamental cycle),
denoted by $Z_{min}$.

The definition of ${\mathcal E^+} $ is motivated by the following fact.
Let $f:(S,0)\to ({\mathbb C},0)$ be a germ of an analytic function. Then the
divisor $(\pi^*(f))$ in $X$ of $f\circ \pi$ can be written as
$D_\pi(f)+S_\pi(f)$, where $D_\pi(f)$, called the compact part of $(\pi^*(f))$,
 is supported on $E$, and $S_\pi(f)$ is the strict
transform by $\pi $ of $\{f=0\}$.  The collection of compact parts (when $f$
runs over ${\mathcal O}_{S,0}$) forms a
semi-group too, it will be denoted by ${\mathcal A^+}$. It is
a sub-semi-group of ${\mathcal E^+} $
(since $(\pi^*(f))\cdot E_i)=0$ and $(S_\pi(f)\cdot E_i)\geq 0$ for all $i$).
The subset ${\mathcal A}^+\setminus \{0\}$ also has
a unique minimal element $Z_{max}$,
the {\it maximal ideal divisor}. It is
the divisor of the generic hyperplane section.  By definitions
$Z_{min}\leq Z_{max}$.

For rational singularities one has ${\mathcal A}^+={\mathcal E}^+ $ (hence $Z_{max}=
Z_{min}$ too). But, in general, these equalities do not hold.
The fundamental cycle $Z_{min}$ can be obtained by Laufer's (combinatorial)
algorithm (cf. \cite{Laufer}), but the structure of ${\mathcal A}^+$ (and even
of $Z_{max}$ too) can be very difficult, it depends
essentially on the analytic structure of $(S,0)$.

\subsection{(Milnor) open books.}\label{ss:ob}
Assume that $f:(S,0)\to ({\mathbb C},0)$ defines an isolated singularity.
Let $M$ be the link of $(S,0)$ and $L_f:=f^{-1}(0)\cap M
\subset M$ the (transversal)
intersection of $f^{-1}(0)$ with  $M$. Then the Milnor fibration of $f$
defines an open book decomposition of $M$ with binding $L_f$. One has the
following facts:

\begin{enumerate}\label{en:1}
\item\label{en:1.1} For any $f$,
consider an embedded good resolution $\pi$ of the pair $(S,f^{-1}(0))$.
Then the strict transform $S_\pi(f)$ intersects $E$ transversally,
and the number of intersection points $(S_\pi(f),E_i)$ (i.e. the number of
binding components associated with $E_i$) is exactly $-(D_\pi(f),E_i)$.
Since the intersection form is negative definite, the collection of
binding components $\{(S_\pi(f),E_i)\}_{i=1}^n$ and $D_\pi(f)\in {\mathcal
  A^+}$ determine each other perfectly.

Moreover, by classical results  of Stallings and Waldhausen, the
(topological type of the) binding  $L_f\subset M$ determines
completely the open book up to an isotopy, provided that $M$ is a
rational homology sphere. (\cite[page 34]{EN} provides two
different arguments for this fact, one of them based on \cite{BL},
the other one on \cite{W}. For counterexamples for the statement in
the general situation, see e.g. \cite{41}.)

Notice that the classification of all the (Milnor) open books associated
with a \textit{fixed analytic type} of $(S,0)$ and analytic
germs $f\in{\mathcal O}_{S,0}$ can be a very difficult problem (in fact, as
difficult  as the determination of ${\mathcal A}^+$).

\item\label{en:1.2} Therefore, from a topological points of view, it is more
  natural to consider the open books of  all the analytic
  germs associated with {\it all the analytic structures} supported by the
topological type of $(S,0)$.

Notice that for a fixed topological type of $(S,0)$, in any (negative definite)
plumbing graph of $M$ one can also define the cone ${\mathcal E^+} $.
The point is that for any non-zero element $D$ of  ${\mathcal E^+} $ there is a
\textit{convenient}
 analytic structure on $(S,0)$ and an  analytic germ $f$, such that
the plumbing graph can be identified with a dual resolution graph (which
serves as an embedded resolution graph for the pair $(S,f^{-1}(0))$ too),
and $D$ is the compact part $D_\pi(f)$, see \cite{P,NP}.  Hence, changing the analytic
structure of $(S,0)$, we fill by the collections  ${\mathcal A^+}$
all the semi-group ${\mathcal E^+} $.

In particular, for any $Z\in{\mathcal E}^+\setminus \{0\}$, there is an open
book  decomposition (well-defined up to an isotopy) realized as Milnor open
book (by a convenient choice of the analytic objects).

\item\label{en:1.3} For any fixed analytic type $(S,0)$, the open book
associated with $Z_{max}$ is the Milnor fibration of the generic hyperplane
section, in particular this open book is (resolution) graph-independent.
Similarly, for a fixed topological type of $(S,0)$,
the open book associated with $Z_{min}$ is also graph-independent.
It depends only on the topology of the link.
\end{enumerate}

\subsection{Invariants of Milnor open books.}\label{ss:2.3}
Let us fix $M$, a plumbing (or, a dual resolution) graph $\Gamma$.
Let us consider a Milnor open book associated with an element
$Z\in {\mathcal E^+}\setminus \{0\}$, cf. (\ref{en:1}).
In the sequel we will
consider the following numerical invariants of it:

\begin{enumerate}\label{en:2}

\item\label{en:2.1} The number of binding components $\beta(Z)$ is given by
$-(Z,E)$ (which is  $\geq 1$).

\item\label{en:2.2} Let $F$ be the fiber of the open book. It is an
oriented  connected  surface with $-(Z,E)$ boundary components.
Let $g(Z)$ be its genus (the so-called page-genus of the open book)
and $\mu(Z)$ be the first Betti-number of $F$
(the so-called Milnor number). Clearly:
\begin{equation}\label{eq:1}
\mu(Z)=2\cdot g(Z)-1+\beta(Z)=  2\cdot g(Z)-1-(Z,E)\geq 2 g(Z).
\end{equation}

\end{enumerate}

We will also write $\nu_i=(E_i,E-E_i)$, the number of components of
$E-E_i$ meeting $E_i$.

\subsection{The `monotoneity' property.}
The main results of the next sections targets the `monotoneity'
property of invariants listed in (\ref{ss:2.3}).

\begin{definition}\label{def:0}
Assume that for any resolution $\pi$ of $(S,0)$ one has a
map $I_\pi:{\mathcal E^+}\setminus \{0\}\to \Z_{\geq 0}$.
We say that $I=\{I_\pi\}_\pi$ is monotone if
for any two cycles $Z_i\in   {\mathcal E^+}\setminus \{0\}$ ($i=1,2$) with
$Z_1\leq Z_2$ one has $I_\pi(Z_1)\leq I_\pi(Z_2)$
for any $\pi$.
\end{definition}

\begin{remark}\label{re:m}
Assume that the collection of invariants $\{I_\pi\}_\pi$ can be transformed
into (or comes from) an invariant $I$ which associates with any (Milnor)
open book
$\mo$ of the link a non-negative integer.
For any fixed analytic type, let $\mo_{\max}$ be the Milnor
open book associated with
$Z_{max}$ (considered in any resolution). Similarly, for any topological type,
let $\mo_{min}$ be the Milnor
open book associated with $Z_{min}$ (in any resolution
of an analytic structure conveniently chosen); cf. (\ref{en:1})(\ref{en:1.3}).

Then, whenever $\{I_\pi\}_\pi$ is monotone, one has automatically
the next consequences:

\begin{enumerate}\label{cor:2}
\item Fix an analytic singularity $(S,0)$ and consider all the Milnor open
books associated with all isolated holomorphic germs $f\in{\mathcal O}_{S,0}$.
Then the minimum of integers $I(\mo)$  of all these Milnor open books $\mo$ is
realized  by the generic hyperplane section, i.e. by $I(\mo_{max})$.

\item Fix a topological type of a normal surface singularity, and consider
the open books associated with all the isolated holomorphic germs of all the
possible analytic structures supported by the fixed topological type.
Then the minimum of all integers $I(\mo)$  of all these Milnor open books
$\mo$ is  realized  by the open book associated with the
 Artin cycle, i.e. by $I(\mo_{min})$.
\end{enumerate}
\end{remark}

\section{The monotoneity of the genus}

\subsection{The relation between the genus and the Euler-characteristic.}
For any fixed graph $\Gamma$, we consider the `canonical cycle' $K\in
{\mathcal G}\otimes \Q$ defined by the (adjunction formulas)
$(K+E_i,E_i)+2=0$ for all $i$. Then the (holomorphic) Euler-characteristic of
any element $D\in {\mathcal G}$ is given by
\begin{equation}\label{eq:2}
\chi(D):=-\frac{1}{2}(D,D+K)\in \Z.
\end{equation}

\begin{proposition} Fix $Z\in {\mathcal E^+}\setminus \{0\}$. Then
\begin{equation}\label{eq:3}
g(Z)=1+(Z, E)+\chi(-Z).
\end{equation}
\end{proposition}
\begin{proof} For any  $1\leq i\leq n$ consider
$k_i:=-(Z,E_i)$  (the number of binding components associated
with $E_i$). Write also $Z=\sum_im_iE_i$.
Then by the A'Campo's formula (cf. \cite{A'Campo})
 $1-\mu =\sum_{i} (2-\nu_i-k_i)m_i$.
Then use (\ref{eq:1}) and (\ref{eq:2}).\end{proof}

\begin{remark}
Since $\chi(-Z)+\chi(Z)+Z^2=0$, one also has
$g(Z)=1+(Z, E-Z)-\chi(Z).$
Since for any  $Z\in {\mathcal E^+}\setminus \{0\}$ one gets
$Z\geq E$, one has $(Z,E-Z)\geq 0$ too. In particular:
\begin{equation}\label{eq:gchi}
g(Z)\geq 1-\chi(Z).\end{equation}
Recall that \textit{rational} singularities are characterized by
$\chi(Z_{min})=1$ \cite{A}. If additionally,
$(S,0)$ is a \textit{minimal} (i.e. if $Z_{min}=E$), then
$g(Z_{min})=0$. For arbitrary rational germs one has
$g(Z_{min})=(Z_{min},E-Z_{min})\geq 0$. This number, in general, might be
non-zero: e.g. in the case of the $E_8$-singularity it is 1.
Considering arbitrary singularities, $\chi(Z_{min})$ tends to $-\infty$ as the
complexity of the topological type of the germ increases, hence by
(\ref{eq:gchi}) $g(Z_{min})$ tends to infinity too.
\end{remark}

\subsection{The ``virtual genus'' and its positivity.}
The formula (\ref{eq:3}) motivates the following definition. For $D=\sum_i
m_iE_i\in{\mathcal G}$, let $|D|$ be the support $\sum_{i\,:\,m_i\not=0}E_i$
of $D$ and $\#(D)$ the number of connected components of $|D|$.

\begin{definition}\label{def:1}
For  $D\in {\mathcal G}$, $D\geq 0$,
we define the ``virtual genus'' of $D$ by
\begin{equation}\label{eq:4}
g(D)=\#(D)+(D, |D|)+\chi(-D).
\end{equation}
\end{definition}

Since for any  $Z\in {\mathcal E^+}\setminus \{0\}$ one has
$|Z|=E$, and $E$ is connected,
(\ref{eq:4}) extends (\ref{eq:3}). Moreover, for any such
$Z\in {\mathcal E^+}\setminus \{0\}$, by its
definition, $g(Z)\geq 0$.

\begin{theorem}\label{th:1}
The virtual genus of any  $D\in {\mathcal G}$, $D\geq 0$,
   is positive: $g(D)\geq 0$.
\end{theorem}

\begin{proof} Assume that the statement is not true at least for one such a
cycle. Since $g(E_i)=1+E_i^2+\chi(-E_i)=0$, there exist a minimal
cycle $D>0$ with $g(D)<0$. Clearly, we can assume that $|D|$ is connected
(and replacing $\Gamma$ by its subgraph supported on $|D|$) that $|D|=E$.
Write $D=\sum_i m_iE_i$. Hence we have:
\begin{equation}\label{eq:11}
1+(D,E)+\chi(-D)<0.\end{equation}
and, using the notation $\#_i$ for the number of components of $|D-E_i|$:
\begin{equation}\label{eq:12}
\#_i+(D-E_i,|D-E_i|)+\chi(-D+E_i)\geq 0\end{equation}

\noindent for all $E_i$. Since $\chi(A+B)=\chi(A)+\chi(B)-(A,B)$, the two
inequalities can easily be compared. Indeed, first assume that $m_i=1$ for
some $i$. Then $|D-E_i|=E-E_i$ and $\#_i=\nu_i$, hence (\ref{eq:11}) and
(\ref{eq:12}) contradict  each other. Therefore, $m_i\geq 2$ for all $i$.
In that case, $|D-E_i|=E$ and $\#_i=1$, hence (\ref{eq:11}) and (\ref{eq:12})
lead to $(D-E,E_i)\geq 0$ for all $i$. Hence $(D-E,D-E)$ is also
non-negative by summation.
Since the intersection form is negative definite, this implies
$D=E$. This contradicts  the fact that $D$ is non-reduced
(and also with the fact that $g(E)=0$).
\end{proof}

\subsection{The monotoneity of the genus}
The main result of this section is the following inequality:

\begin{theorem}\label{th:2}
Consider two cycles $Z$ and $Z+D$, where $Z\in{\mathcal E^+}\setminus
\{0\}$ and $D\in{\mathcal G}$, $D\geq 0$.
Then the (virtual) genera satisfy $g(Z)\leq g(Z+D)$.
\end{theorem}

\begin{proof} By (\ref{eq:3}), one has
\begin{eqnarray*}\label{eq:16}
g(Z+D)-g(Z)&=&(D,E)+\chi(-D)-(D,Z)\\
&=&g(D)+(D,E-|D|)-\#(D)-(D,Z).
\end{eqnarray*}
If $|D|=E$ then $\#(D)=1$ and $-(D,Z)\geq 1$ (otherwise we would have
$(Z,E_i)=0$
for all $i$, or $Z=0$). If $|D|<E$, then $-(D,Z)\geq 0$ and
$(D,E-|D|)\geq (|D|,E-|D|)\geq \#(D)$ by the connectivity of $\Gamma$.
Hence, in both cases, the right--hand side is $\geq g(D)$.
Since $g(D)\geq 0$ by (\ref{th:1}), the inequality follows.
\end{proof}

\begin{corollary}\label{cor:1} The genus   is monotone:  for any
$Z_1$ and $Z_2$ from ${\cal E}^+\setminus \{0\}$  with
$Z_1\leq Z_2$ one has $g(Z_1)\leq g(Z_2)$. In particular, the statements of
(\ref{re:m}) also hold.
\end{corollary}

\section{The Milnor number and the number of boundary components}

\subsection{The monotoneity of the Milnor number}
If one combines (\ref{eq:1}) and (\ref{eq:3}), one gets for any
$Z\in {\mathcal E^+}\setminus \{0\}$:
\begin{eqnarray}\label{eqn:1}
\mu(Z)&=& 1+(Z,E)+2\cdot \chi(-Z)\\
&=& g(Z)+\chi(-Z).\label{eqn:2}
\end{eqnarray}
Again, we extend the above formula (in a compatible way with  ({\ref{eqn:2}))
for any $D\geq 0$ by considering
the `virtual Milnor number' $\mu(D)$  as
$g(D)+\chi(-D)$, defined via the virtual genus $g(D)$.

Clearly, $\mu(Z)\geq 0$ for any $Z\in {\mathcal E^+}\setminus \{0\}$, since
$\mu(Z)$ stays for a Betti number. Moreover, for any rational graph
$\Gamma$, one has $\min\chi= 0$, hence for them the virtual invariants
satisfy $\mu(D)\geq g(D)\geq 0$ too. The next theorem generalizes this
for a general $\Gamma$.

\begin{theorem}\label{th:mu} Set  $D\in{\mathcal G}$ with  $D\geq
  0$. Then the following inequalities hold:

\begin{enumerate}
\item\label{en:mu.1}  \ \ $\chi(-D)\geq 0$;
\item\label{en:mu.2}  \ \ $\mu(D)\geq g(D)\geq 0$;
\item\label{en:mu.3}  \ \ $\mu(Z+D)\geq \mu(Z)$ \ for any
 $Z\in {\mathcal E^+}\setminus \{0\}$.
\end{enumerate}
\end{theorem}

\begin{proof}
The proof of (\ref{en:mu.1}) is well-known for specialist,
for the convenience of the reader we provide it.
We claim that for any $D>0$ there exists at least one $E_i$ with
$E_i\leq D$ such that $\chi(-D+E_i)\leq \chi(-D)$. This by induction shows
that $\chi(-D)\geq 0$. The proof of the claim runs as follows. Assume that
it is not true for some $D>0$. Then for any $E_i$ from its support one has
$\chi(-D+E_i)\geq \chi(-D)+1$. This is equivalent with $(D,E_i)\geq 0$, hence
by summation one gets $D^2\geq 0$. This implies $D=0$, a contradiction.

(\ref{en:mu.2}) follows from (\ref{eqn:2}),
part  (\ref{en:mu.1}) and (\ref{th:1}). For
(\ref{en:mu.3}) notice that by (\ref{eqn:2})
$$\mu(Z+D)-\mu(Z)=g(Z+D)-g(Z)+\chi(-D)-(Z,D).$$
Notice that $g(Z+D)\geq g(Z)$ by (\ref{th:2}), $\chi(-D)\geq 0$ by
(\ref{en:mu.1}), and $-(Z,D)\geq 0$ since  $Z\in {\mathcal E^+}$.
\end{proof}

\begin{corollary}\label{cor:22} The Milnor number   is monotone:  for any
$Z_1$ and $Z_2$ from ${\cal E}^+\setminus \{0\}$  with
$Z_1\leq Z_2$ one has $\mu(Z_1)\leq \mu(Z_2)$. In particular, the statements of
(\ref{re:m}) also hold for $\mu$.
\end{corollary}

\subsection{The number of binding components}
Recall that the number of binding
components of the open book associated with
some $Z\in {\mathcal E^+}\setminus \{0\}$ is $\beta(Z)=-(Z,E)$.
We wish to understand
the variation of this number in the realm of (Milnor) open books with
page-genus fixed. In order to do this, let us consider the following
subsets of  ${\mathcal E^+}$:
$${\mathcal E}^+_{min}:=\{Z\, |\, g(Z)=g(Z_{min})\}, \ \mbox{and} \
{\mathcal E}^+_{g=a}:=\{Z\, |\, g(Z)=a\},$$
where $a\in\Z$.  Since $\mu(Z)-\beta(Z)=2g(Z)-1$, we get:

\begin{lemma}\label{lem:bn}
For any $a$,  the restrictions of $\mu$ and $\beta$ to
${\mathcal E}^+_{g=a}$ take their minima on the same elements of
${\mathcal E}^+_{g=a}$. In particular,
the restriction of $\mu$ (resp. of $\beta$) on
${\mathcal E}^+_{min}$ is $\mu(Z_{min})$ (resp. $\beta(Z_{min})$).
\end{lemma}

\section{Application to the canonical contact structure of the link}\label{s5}

Our application targets the invariants
sg$_{an}(\xi_{can})$, bn$_{an}(\xi_{can})$ and $\mathfrak{n}_{an}(\xi_{can})$;
for notations, see Introduction. Indeed,
the previous results read as follows:
$$\mbox{sg}_{an}(\xi_{can})=g(Z_{min});$$
$$\mbox{bn}_{an}(\xi_{can})=\beta(Z_{min}); $$
$$ \mathfrak{n}_{an}(\xi_{can})=\mu(Z_{min})-1.$$
In particular,
 $$\mathfrak{n}_{an}(\xi_{can})-\mbox{bn}_{an}(\xi_{can})=
2\cdot \mbox{sg}_{an}(\xi_{can})-2.$$
These facts answer some of the questions of  \cite{burak-etnyre}, section 8,
at least in the realm of Milnor open books.

{}

\end{document}